\documentclass[12pt]{article}
\usepackage{bbm}
\usepackage{latexsym}
\usepackage{mathrsfs}
\usepackage{graphicx}
\usepackage{subfigure}
\usepackage{amssymb}
\usepackage{amsmath}
\usepackage{amsthm}
\usepackage{color}
\usepackage{cite}
\usepackage{float}
\usepackage{indentfirst}
\usepackage{anysize}\marginsize{45mm}{45mm}{40mm}{50mm}
\usepackage{caption}
\usepackage{float}
\usepackage{multirow}

\usepackage{enumerate}

\usepackage{latexsym, bm}
\newtheoremstyle{lemma}{\topsep}{\topsep}%
     {}
     {}
     {\bfseries}
     {}
     {0.1em}
     {\thmname{#1}\thmnumber{ #2}\thmnote{ #3}}
\theoremstyle{lemma}  

\newtheorem{theorem}{Theorem}     
\newtheorem{lemma}[theorem]{Lemma}

\newtheorem{definition}{Definition}

\numberwithin{equation}{section}

\title{Unpaired many-to-many disjoint path cover of balanced hypercubes\thanks{This research was partially supported by the National Natural Science Foundation of China (Nos. 11801061 and 11761056) and the Chunhui Project of Ministry of Education (No. Z2017047).}}

\author{ Huazhong L\"{u}$^{1}$\thanks{Corresponding author.} and Tingzeng Wu$^{2}$\\
{\small $^{1}$School of Mathematical Sciences,} \\
{\small University of Electronic Science and Technology of China,}\\
{\small Chengdu, Sichuan 610054, P.R. China}\\
{\small E-mail: lvhz08@lzu.edu.cn}\\
{\small $^{2}$School of Mathematics and Statistics, Qinghai Nationalities University, }\\
{\small Xining, Qinghai 810007, P.R. China} \\
{\small E-mail: mathtzwu@163.com}\\}
\date{}
\begin{document}

\maketitle
\begin{abstract}

The balanced hypercube $BH_n$, a variant of the hypercube, was proposed as a desired interconnection network topology. It is known that $BH_n$ is bipartite. Assume that $S=\{s_1,s_2,\cdots,s_{2n-2}\}$ and $T=\{t_1,t_2,\cdots,t_{2n-2}\}$ are any two sets of vertices in different partite sets of $BH_n$ ($n\geq2$). It has been proved that there exists paired 2-disjoint path cover of $BH_n$. In this paper, we prove that there exists unpaired $(2n-2)$-disjoint path cover of $BH_n$ ($n\geq2$) from $S$ to $T$, which improved some known results. The upper bound $2n-2$ of the number of disjoint paths in unpaired $(2n-2)$-disjoint path cover is best possible.

\vskip 0.1 in

\end{abstract}

\section{Introduction}
The interconnection network (network for short) plays an important role in massively parallel and distributed systems \cite{Leighton}. Linear arrays and rings are two fundamental networks. Since some parallel applications such as those in image and signal processing are originally designated on an array architecture, it is important to have effective path embedding in a network \cite{Bhandarkar,Wani}. To find parallel paths among vertices in networks is one of the most central issues concerned with efficient data transmission \cite{Leighton}. Parallel paths in networks are usually studied with regard to disjoint paths in graphs. Moreover, algorithms designed on linear arrays or rings can be efficiently simulated in a topology containing paths or cycles, so path and cycle embedding properties of networks have been widely studied \cite{Chen,Dong,Dybizbanski,Fan,Gould,Jo,Kim,Tsai,Wang,Xu,Yang}.

In disjoint path cover problems, the many-to-many disjoint path cover problem is the most generalized one\cite{Park2}. Assume that $S=\{s_1, s_2,\cdots, s_k\}$ and $T=\{t_1, t_2,\cdots, t_k\}$ are two sets of $k$ sources and $k$ sinks in a graph $G$, respectively, the {\em many-to-many $k$-disjoint path cover} ($k$-DPC for short) problem is to determine whether there exist $k$ disjoint paths $P_1,P_2,\cdots,P_k$ in $G$ such that $P_i$ joins $s_i$ to $t_{\psi(i)}$ for each $i\in\{1,2,\cdots,k\}$ and $V(P_1)\cup\cdots \cup V(P_k)=V(G)$, where $\psi$ is a permutation on the set $\{1,2,\cdots,k\}$.
The $k$-DPC is called {\em paired} if $\psi$ is the identical permutation
and {\em unpaired} otherwise. Interestingly, the $k$-DPC problem is closely related to the well-known Hamiltonian path problem in graphs. In fact, a 1-DPC of a network is indeed a Hamiltonian path between any two vertices.

The performance of the famous hypercube network is not optimum in all aspects, accordingly, many variants of the hypercube have been proposed. The balanced hypercube, proposed by Wu and Huang \cite{Wu}, is such the one of the most popularity. The special property of the balanced hypercube, which other hypercube variants do not have, is that each processor has a backup processor that shares the same neighborhood. Thus tasks running on a faulty processor can be shifted to its backup one \cite{Wu}. With such novel properties above, different aspects of the balanced hypercube were studied extensively,
including path and cycle embedding issues \cite{Hao,Li,Lu4,Xu,Yang,Zhou}, connectivity \cite{Lu2,Yang3},
matching preclusion\cite{Lu}, and symmetric properties \cite{Zhou2,Zhou3}.

Recently, Cheng el al. \cite{Cheng} proved that the balanced hypercube $BH_n$ ($n\geq1$) has a paired 2-DPC, which is a generalization of Hamiltonian laceability of the balanced hypercube \cite{Xu}. To the best of our knowledge, there is no literature on $k$-DPC in the balanced hypercube when $k\geq3$. In this paper, we will consider the problem of unpaired $k$-DPC of the balanced hypercube.

The rest of this paper is organized as follows. In Section 2, some definitions and lemmas are presented. The main result of this paper is shown in Section 3. Conclusions are given in Section 4.

\section{Preliminaries and some lemmas}

A network is usually modeled by a simple undirected graph, where vertices represent processors and edges represent links between processors. Let $G=(V(G),E(G))$ be a graph, where $V(G)$ and $E(G)$ are its vertex-set and edge-set, respectively. The number of vertices of $G$ is denoted by $|V(G)|$. The set of vertices adjacent to a vertex $v$ is called the {\em neighborhood} of $v$, denoted by $N_G(v)$, the subscript will be omitted when the context is clear. A path $P$ in $G$ is a sequence of distinct vertices so that there is an edge joining each pair of consecutive vertices. If $P=v_0v_1\cdots v_{k-1}$ and $k\geq3$, then the graph $C=P+v_0v_{k-1}$ is called a {\em cycle}. A path (resp. cycle) containing all vertices of a graph $G$ is called a {\em Hamiltonian path} (resp. {\em cycle}). A graph admits a Hamiltonian cycle is a {\em Hamiltonian graph}. A Hamiltonian bipartite graph $G$ is {\em Hamiltonian laceable} if, for any two vertices $u$ and $v$ from different partite sets, there exists a Hamiltonian path between $u$ and $v$. For other standard graph notations not defined here please refer to \cite{Bondy}.
\vskip 0.05 in

The definitions of the balanced hypercube are given as follows.
\vskip 0.0 in

\begin{definition}{\bf .}\label{def1}\cite{Wu} An $n$-dimension balanced hypercube $BH_{n}$ contains
$4^{n}$ vertices $(a_{0},$ $\ldots,a_{i-1},$
$a_{i},a_{i+1},\ldots,a_{n-1})$, where $a_{i}\in\{0,1,2,3\},$ $0\leq
i\leq n-1$. Any vertex $v=(a_{0},\ldots,a_{i-1},$
$a_{i},a_{i+1},\ldots,a_{n-1})$ in $BH_{n}$ has the following $2n$ neighbors:

\begin{enumerate}
\item $((a_{0}+1)$ mod $
4,a_{1},\ldots,a_{i-1},a_{i},a_{i+1},\ldots,a_{n-1})$,\\
      $((a_{0}-1)$ mod $ 4,a_{1},\ldots,a_{i-1},a_{i},a_{i+1},\ldots,a_{n-1})$, and
\item $((a_{0}+1)$ mod $ 4,a_{1},\ldots,a_{i-1},(a_{i}+(-1)^{a_{0}})$ mod $
4,a_{i+1},\ldots,a_{n-1})$,\\
      $((a_{0}-1)$ mod $ 4,a_{1},\ldots,a_{i-1},(a_{i}+(-1)^{a_{0}})$ mod $
      4,a_{i+1},\ldots,a_{n-1})$.
\end{enumerate}
\end{definition}

The first coordinate $a_{0}$ of the vertex
$(a_{0},\ldots,a_{i},\ldots,a_{n-1})$ in $BH_{n}$ is defined as the {\em inner index}, and
$a_{i}$ $(1\leq i\leq n-1)$ are {\em $i$-dimensional indices}.

\vskip 0.0 in

The recursive definition of the balanced hypercube is presented as follows.

\begin{definition}{\bf .}\label{def2}\cite{Wu}
\begin{enumerate}
\item $BH_{1}$ is a $4$-cycle, whose vertices are labelled
by $0,1,2,3$ clockwise.
\item $BH_{k+1}$ is constructed from $4$ $BH_{k}$s, which
are labelled by $BH^{0}_{k}$, $BH^{1}_{k}$, $BH^{2}_{k}$,
$BH^{3}_{k}$. For any vertex in $BH_{k}^{i}(0\leq i\leq 3)$, its new labelling in $BH_{k+1}$ is $(a_{0},a_{1},\ldots,a_{k-1},i)$, and it has two new neighbors:
\begin{enumerate}
\item[a)] $BH^{i+1}_{k}:((a_{0}+1)$ mod $4,a_{1},\ldots,a_{k-1},(i+1)$ mod $4)$ and

$((a_{0}-1)$ mod $4,a_{1},\ldots,a_{k-1},(i+1)$ mod $4)$ if $a_{0}$ is even.

\item[b)] $BH^{i-1}_{k}:((a_{0}+1)$ mod $4,a_{1},\ldots,a_{k-1},(i-1)$ mod $4)$ and

$((a_{0}-1)$ mod $4,a_{1},\ldots,a_{k-1},(i-1)$ mod $4)$ if $a_{0}$ is odd.
\end{enumerate}

\end{enumerate}
\end{definition}

$BH_{1}$ is shown in Fig. \ref{g1} (a). The standard layout of $BH_{2}$ is shown in Fig. \ref{g1} (b) and the ring-like layout is shown in Fig. \ref{g1} (c).

The following basic properties of the balanced hypercube will be applied in the main result of this paper.

\begin{lemma}\label{bipartite}\cite{Wu}{\bf.}
$BH_{n}$ is bipartite.
\end{lemma}

By Lemma \ref{bipartite}, let $V_0$ and $V_1$ be the bipartition of $BH_{n}$, where $V_0$ contains all vertices of $BH_n$ with even inner indices, and $V_1$ contains all vertices of $BH_n$ with odd inner indices.

\begin{lemma}\label{transitive}\cite{Wu,Zhou}{\bf.}
$BH_{n}$ is vertex-transitive and edge-transitive.
\end{lemma}

\begin{lemma}\label{neighbor}\cite{Wu}{\bf.}
Vertices $u=(a_{0},a_{1},\ldots,a_{n-1})$ and
$v=((a_{0}+2)$ mod 4, $a_{1},\ldots,a_{n-1})$ in $BH_{n}$ have the same neighborhood.
\end{lemma}

\begin{figure}
\centering
\includegraphics[width=150mm]{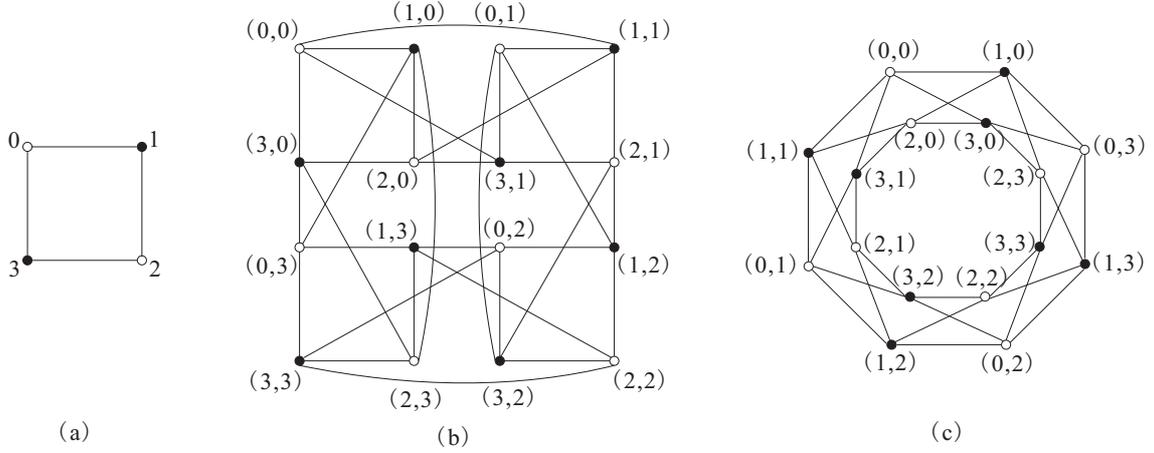}
\caption{$BH_{1}$ and $BH_{2}$.} \label{g1}
\end{figure}

Assume that $u$ is a neighbor of $v$ in $BH_n$. If $u$ and $v$ differ only from the inner index, then $uv$ is called a $0$-{\em dimension edge}, and $u$ and $v$ are mutually called 0-dimension neighbors. Similarly, if $u$ and $v$ differ from the $j$-th outer index ($1\leq j\leq n-1$), $uv$ is called a $j$-{\em dimension edge}, and $u$ and $v$ are mutually called $j$-dimension neighbors. The set of all $k$-dimension edges of $BH_n$ is denoted by $E_k$ for each $k\in\{0,\cdots,n-1\}$, and the subgraph of $BH_{n}$ obtained by deleting $E_{n-1}$ is written by $BH^{i}_{n-1}$, where $0\leq i \leq 3$. Obviously, each of $BH^{i}_{n-1}$ is isomorphic to $BH_{n-1}$.

\begin{lemma}\label{lace}
{\rm \cite{Li}}{\bf.} The balanced hypercube $BH_{n}$ is Hamiltonian laceable for all $n\geq 1$.
\end{lemma}

%


\begin{lemma}{\rm \cite{Cheng}}\label{2paths}{\bf .}  Let $u,x\in V_{0}$ and $v,y\in V_{1}$. Then there exist two vertex-disjoint paths $P$ and $Q$ such that: (1) $P$ connects $u$ to $v$, (2) $Q$ connects $x$ to $y$, (3) $V(P)\cup V(Q)=V(BH_{n})$.
\end{lemma}


\section{Main results}

Because of the recursive structure of the balanced hypercube, we use induction to prove the main result. We start with the following useful notation.

Let $S=\{s_1,s_2,\cdots,s_{2n-2}\}$ and $T=\{t_1,t_2,\cdots,t_{2n-2}\}$ such that $S\subset V_0$ and $T\subset V_1$, $n\geq2$. For convenience, let $s_j^d$ denote the $d$-dimensional index of the vertex $s_j$, where $1\leq j\leq 2n-2$ and $d\in\{1,2,\cdots,n-1\}$.

\begin{lemma}\label{split}{\bf.} There exists a dimension $d\in\{0,1,\cdots,n-1\}$ such that by splitting $BH_n$ ($n\geq3$) along dimension $d$, $|V(BH_{n-1}^i)\cap S|\leq2n-4$ for each $0\leq i\leq n-1$.
\end{lemma}
\begin{proof} We first consider $BH_3$. Suppose on the contrary that for each $d\in\{1,2\}$, by splitting $BH_3$ along dimension $d$, there exists some $i$ ($0\leq i\leq 3$) such that $|V(BH_{2}^i)\cap S|\geq3$. Note that $S=\{s_1,s_2,s_3,s_4\}$, then $s_j^d$ ($j=1,2,3,4$) takes at most two values for some $d=1$ or $2$. If each of $s_j^d$, say $s_j^2$, takes exact one value, combining with $S\subset V_0$, then two vertices in $S$ have the same coordinates, which is a contradiction. So we assume that there are exactly two values of $s_j^d$ for each $d=1,2$. Then three of $s_j^d$ take one common value for $j=1,2,3,4$. Thus, two of $s_1,s_2,s_3$ and $s_4$, say $s_1$ and $s_2$, have the same 1-dimensional and 2-dimensional indices, and distinct inner indices. Observe that exactly one of the 1-dimensional and 2-dimensional indices of $s_3$ (resp. $s_4$) is different from that of $s_1$ and $s_2$. Thus, by splitting $BH_3$ along dimension $0$, $s_1$ and $s_2$ are in the same $BH_{2}^i$ for some $0\leq i\leq 3$, and $s_3,s_4\not\in V(BH_{2}^i)$.

Now we consider $BH_n$ for $n\geq4$. Clearly, $S=\{s_1,s_2,\cdots,s_{2n-2}\}$. Suppose on the contrary that $s_j^d$ ($j=1,2,\cdots,2n-2$) take at most two values and $2n-3$ of which take one common value for each $d=1,2,\cdots,n-1$. We can consider the coordinates (except inner index) of vertices in $S$ as row vectors, forming a $(2n-2)\times (n-1)$ matrix $M$. Thus, there exists at least three equal rows of $M$, indicating that there are three vertices in $S$ differing only the inner indices. Note that the inner indices of vertices in $S$ take only two values. So there are two vertices in $S$ with the same coordinates, which is a contradiction. This completes the proof.
\end{proof}

By the above lemma, there exists a dimension $d\in\{0,1,2,\cdots,n-1\}$, such that by splitting $BH_n$ along dimension $n-1$, each $BH_{n-1}^i$ contains at most $2n-4$ vertices in $S$, $0\leq i\leq 3$. We may assume that $d=n-1$ in the remaining paper. Let $S_i=V(BH_{n-1}^i)\cap S$ and $T_i=V(BH_{n-1}^i)\cap T$ and let $D_i=|T_i|-|S_i|$. We have the following lemma.

\begin{lemma}{\bf .}\label{D_i} There exists some $i\in\{0,1,2,3\}$ such that $|S_{i}|\geq|T_{i}|$ and $|S_{i+1}|\leq|T_{i+1}|$. Furthermore, $D_{i+1}+D_{i+2}\geq0$.
\end{lemma}
\begin{proof} By the ring-like structure of $BH_n$, it is obvious that there exists some $i\in\{0,1,2,3\}$ such that $|S_{i}|\geq|T_{i}|$ and $|S_{i+1}|\leq|T_{i+1}|$, implying that $D_i\leq 0$ and $D_{i+1}\geq 0$. Obviously, if $D_{i}\geq0$ for each $0\leq i\leq3$, we are done. Next we distinguish the following cases.

{\bf Case 1.} There exists exactly one integer $i\in\{0,1,2,3\}$ such that $D_{i+1}\geq 0$ or $D_{i+1}\leq 0$. We only consider $D_{i+1}\geq 0$ since the same argument applies to $D_{i+1}\leq 0$. Clearly, $D_i<0, D_{i+2}<0$ and $D_{i+3}<0$. Noting $|S|=\sum_{i=0}^{3}|S_i|$, $|T|=\sum_{i=0}^{3}|T_i|$ and $|S|=|T|$, it follows that $D_{i+1}+D_{i+2}\geq0$.

{\bf Case 2.} There exist exactly two distinct integers $i,j\in\{0,1,2,3\}$ such that $D_{i}\geq 0$ and $D_j\geq 0$, where $i\neq j$. By the ring-like structure of $BH_n$, there are two essentially distinct cases by relative positions of $BH_{n-1}^i$ and $BH_{n-1}^j$. We further distinguish the following two cases.

{\bf Case 2.1.} $j=i+1$. That is, $D_i\geq0$ and $D_{i+1}\geq0$, and $D_{i+2}<0$ and $D_{i+3}<0$. Obviously, $D_i+D_{i+1}\geq0$.

{\bf Case 2.2.} $j=i+2$. That is, $D_i\geq0$ and $D_{i+2}\geq0$, and $D_{i+1}<0$ and $D_{i+3}<0$. Similarly, $D_i+D_{i+1}\geq0$ or $D_{i+2}+D_{i+3}\geq0$ holds. This completes the proof.
\end{proof}

\begin{theorem}{\bf .} There exists an unpaired $(2n-2)$-DPC joining $S$ and $T$.
\end{theorem}
\begin{proof} We prove by induction on $n$. By Lemma \ref{2paths}, the theorem obviously holds for $n=2$. Suppose the statement holds for $n-1$ with $n\geq3$. Next we consider $BH_n$. By Lemmas \ref{split} and \ref{D_i}, we split $BH_n$ into four $BH_{n-1}$s along dimension $n-1$ such that $|S_i|\leq 2n-4$ for each $i\in\{0,1,2,3\}$, and for some $i\in\{0,1,2,3\}$, say $i=0$, we have $|S_{0}|\geq|T_{0}|$, $|S_{1}|\leq|T_{1}|$ and $D_{1}+D_{2}\geq0$. We distinguish the following cases.

\noindent{\bf Case 1.} $S_i\neq\emptyset$ or $T_i\neq\emptyset$ for each $0\leq i\leq3$.

\noindent{\bf Case 1.1} $|T_1|\leq 2n-4$. Choose $|T_1|-|S_1|$ white vertices from $V(BH_{n-1}^1)\setminus S_1$ to generate a set $A_1$. By the induction hypothesis, we can obtain $|T_1|$-DPC of $BH_{n-1}^1$ from $S_1\cup A_1$ to $T_1$. Note that $A_1=\emptyset$ if $|T_1|=|S_1|$. Let $A_1=\{a_1,a_2,\cdots,a_{|T_1|-|S_1|}\}$ and let $B_2=\{b_1,b_2,\cdots,b_{|T_1|-|S_1|}\}$ such that $b_j$ is an $(n-1)$-dimensional neighbor of $a_j$ for each $j\in\{1,2,\cdots,|T_1|-|S_1|\}$. Since $D_{1}+D_{2}\geq0$, we have $|S_2|\leq|T_2\cup B_2|$. We consider the following two conditions in terms of the cardinality of $|T_2\cup B_2|$ (This argument will be used repeatedly in the remaining proof).

\begin{enumerate}
\item If $|T_2\cup B_2|\leq 2n-4$, then we choose $|T_2\cup B_2|-|S_2|$ white vertices from $V(BH_{n-1}^2)\setminus S_2$ to generate a set $A_2$ (Since there are exactly $2^{2n-3}$ white vertices in $BH_{n-1}^1$, combining with $2^{2n-3}\geq 3\ast(2n-4)$ whenever $n\geq3$, we can always choose $A_2$ to make $B_2\cap T_2=\emptyset$). By the induction hypothesis, we can obtain $|T_2\cup B_2|$-DPC of $BH_{n-1}^2$ from $T_2\cup B_2$ to $S_2\cup A_2$. Note also that $A_2=\emptyset$ if $|T_2\cup B_2|=|S_2|$. Let $A_2=\{a_1',a_2',\cdots,a_{|B_2\cup T_2|-|S_2|}'\}$ and let $B_3=\{b_1',b_2',\cdots,b_{|B_2\cup T_2|-|S_2|}'\}$ such that $b_k'$ is an ($n-1$)-dimensional neighbor of $a_k'$ for each $k\in\{1,2,\cdots,|B_2\cup T_2|-|S_2|\}$.

\item If $|T_2\cup B_2|\geq 2n-3$, then we choose $2n-4-|S_2|$ white vertices from $V(BH_{n-1}^2)\setminus S_2$ to generate a set $A_2$ (It is not hard to choose $A_2$ to make $B_2\cap T_2=\emptyset$ when $n=3$; the argument is similar to that in above condition when $n=4$). We then arbitrarily choose $2n-4$ vertices from $T_2\cup B_2$ to form a set $T_2'$. By the induction hypothesis, we can obtain $|T_2'|$-DPC of $BH_{n-1}^2$ from $T_2'$ to $S_2\cup A_2$. Note that $A_2=\emptyset$ if $|S_2|=2n-4$. Since $1\leq|T_2\cup B_2|-|T_2'|\leq2$, there are at most two vertices of $(T_2\cup B_2)\setminus T_2'$ on at most two paths of $|T_2'|$-DPC of $BH_{n-1}^2$. Suppose without loss of generality that $y_1,y_2\in (T_2\cup B_2)\setminus T_2'$. Additionally, suppose that $y_1$ and $y_2$ are on the same path $P$ of $|T_2'|$-DPC (the proof of $y_1$ and $y_2$ on different paths of $|T_2'|$-DPC is similar to those on the same path). Let $x$ and $y$ be the endpoints of $P$, where $x\in V_0$ and $y\in V_1$. We may assume that $x, y_1, y_2$ and $y$ lie on $P$ sequentially. So there exists a neighbor $x_1$ (resp. $x_2$) of $y_1$ (resp. $y_2$) from $y$ to $x$. Thus, the path $P$ can be separated into three vertex-disjoint sections $P_1,P_2$ and $P_3$, where $P_1$ is from $x$ to $y_1$, $P_2$ is from $x_1$ to $y_2$ and $P_3$ is from $x_2$ to $y$. Therefore, $P_1,P_2$ and $P_3$ together with paths in $|T_2'|$-DPC except $P$ form a $|T_2\cup B_2|$-DPC of $BH_{n-1}^2$ from $T_2\cup B_2$ to $S_2\cup A_2\cup\{x_1,x_2\}$. We may assume that $A_2=\{a_1',a_2',\cdots,a_{2n-4-|S_2|}'\}$. Let $B_3=\{b_1',b_2',\cdots,b_{2n-4-|S_2|}'\}\cup\{y_1',y_2'\}$ such that $b_k'$ is an ($n-1$)-dimensional neighbor of $a_k'$ for each $k\in\{1,2,\cdots,2n-4-|S_2|\}$ and $y_1'$ and $y_2'$ are ($n-1$)-dimensional neighbors of $x_1'$ and $x_2'$, respectively.

\end{enumerate}

It is not hard to see that $|S_3|\leq|T_3\cup B_3|$.

If $|T_3\cup B_3|\leq 2n-4$, analogous to Condition (1), then we choose $|T_3\cup B_3|-|S_3|$ white vertices from $V(BH_{n-1}^3)\setminus S_3$ to generate a set $A_3$. By the induction hypothesis, we can obtain $|T_3\cup B_3|$-DPC of $BH_{n-1}^3$ from $T_3\cup B_3$ to $S_3\cup A_3$. Note that $A_3=\emptyset$ if $|T_3\cup B_3|=|S_3|$. Let $A_3=\{a_1'',a_2'',\cdots,a_{|B_3\cup T_3|-|S_3|}''\}$ and let $B_0=\{b_1'',b_2'',\cdots,b_{|B_3\cup T_3|-|S_3|}''\}$ such that $b_l''$ is an ($n-1$)-dimensional neighbor of $a_l''$ for each $l\in\{1,2,\cdots,|B_3\cup T_3|-|S_3|\}$. Clearly, $|T_0\cup B_0|=|S_0|$. Additionally, $|S_0|\leq 2n-4$. By the induction hypothesis, we can obtain $|S_0|$-DPC of $BH_{n-1}^0$ from $T_3\cup B_3$ to $S_0$. Thus, we can obtain $(2n-2)$-DPC of $BH_{n}$ from $S$ to $T$ (see Fig. \ref{g2}).

\begin{figure}
\centering
\includegraphics[width=60mm]{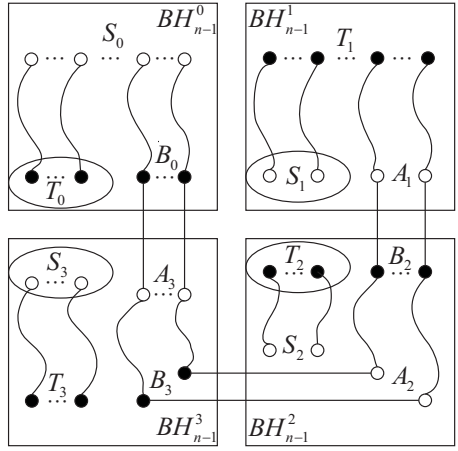}
\caption{Illustration of Case 1.1.} \label{g2}
\end{figure}

If $|T_3\cup B_3|\geq 2n-3$, analogous to Condition (2), then we can obtain $|T_3\cup B_3|$-DPC of $BH_{n-1}^3$ from $T_3\cup B_3$ to $S_3\cup A_3$, where $A_3$ contains white endpoints of $|T_3\cup B_3|$-DPC of $BH_{n-1}^3$. Let $B_0$ be a mirroring set of $A_3$ such that each vertex in $A_3$ has exactly one $(n-1)$-dimensional neighbor in $B_0$, and vice versa. Clearly, $|T_0\cup B_0|=|S_0|$ and $|S_0|\leq 2n-4$. By the induction hypothesis, we can obtain $|S_0|$-DPC of $BH_{n-1}^0$ from $T_3\cup B_3$ to $S_0$. Thus, we can obtain $(2n-2)$-DPC of $BH_{n}$ from $S$ to $T$.

\noindent{\bf Case 1.2.} $|T_1|\geq2n-3$. Choose $2n-4-|S_1|$ white vertices from $V(BH_{n-1}^1)\setminus S_1$ to generate a set $A_1$. By regarding $|T_1|$ as $T_2\cup B_2$, and $S_1\cup A_1$ as $S_2\cup A_2$, in Condition (2), the proof is quite analogous to that of Case 1.1.

\noindent{\bf Case 2.} $S_j=\emptyset$ and $T_j=\emptyset$ for exactly one $j\in\{0,1,2,3\}$. By symmetry of $BH_n$, it suffices to consider that $j=1$ or 2. We distinguish the following two cases.

\noindent{\bf Case 2.1.} $j=2$. If $D_1>0$, then the proof is analogous to that of Case 1. So we assume that $D_1=0$, that is, $|T_{1}|=|S_{1}|$. By Lemma \ref{split}, we have $|S_1|\leq 2n-4$. So $|T_{1}|\leq 2n-4$. By the induction hypothesis, there exists $|T_{1}|$-DPC of $BH_{n-1}^1$ from $T_1$ to $S_1$. In addition, recall that $|S_2|=|T_2|$, combining with $|D_0|\leq0$, then we have $|T_3|\geq |S_3|$. By adopting the same argument in Case 1.1, we can obtain $|T_3|$-DPC of $BH_{n-1}^3$ from $T_3$ to $S_3\cup A_3$, where $A_3$ contains white endpoints of $|T_3|$-DPC. Note that $A_3=\emptyset$ if $|T_3\cup B_3|=|S_3|$. Let $B_0\in V(BH_{n-1}^0)$ be a mirroring set of $A_3$ such that each vertex of $B_0$ is an ($n-1$)-dimensional neighbor of a vertex in $A_3$. We can obtain $|S_0|$-DPC of $BH_{n-1}^0$ from $T_0\cup B_0$ to $S_0$.

What we have already shown is that there exists a $(2n-2)$-DPC of $BH_{n}-V(BH_{n-1}^2)$. In what follows, we shall make some changes to $(2n-2)$-DPC of $BH_{n}-V(BH_{n-1}^2)$, yielding a $(2n-2)$-DPC of $BH_n$.

Suppose that $xy\in E(BH_{n-1}^0)$ is an edge on one of a path, say $P_0$, of $(2n-2)$-DPC of $BH_{n}-V(BH_{n-1}^2)$ such that $s_0,y,x$ and $t_0$ lie sequentially on $P_0$, where $s_0$ and $t_0$ are the endpoints of $P_0$. Note that $s_0\in S_0$ and $t_0$ may or may not belong to $T_0$. Let $y_1$ and $x_3$ be $(n-1)$-dimensional neighbors of $x$ and $y$, respectively. If $y_1\not\in T_1$ and $x_3\not\in S_3\cup A_3$, then $y_1$ (resp. $x_3$) is an internal vertex of a path $P_1$ (resp. $P_3$) of $|T_{1}|$-DPC (resp. $|T_3|$-DPC) of $BH_{n-1}^1$ (resp. $BH_{n-1}^3$). Let $x_1$ (resp. $y_3$) be a neighbour of $y_1$ (resp. $x_3$) on $P_1$ (resp. $P_3$). Suppose without loss of generality that $s_1,y_1,x_1$ and $t_1$ (resp. $s_3,y_3,x_3$ and $t_3$) lie sequentially on $P_1$ (resp. $P_3$), where $s_1$ and $t_1$ (resp. $s_3$ and $t_3$) are the endpoints of $P_1$ (resp. $P_3$). Let $y_2$ and $x_2$ be $(n-1)$-dimensional neighbors of $x_1$ and $y_3$, respectively. By Lemma \ref{lace}, there exists a Hamiltonian path $P_2$ of $BH_{n-1}^2$ from $y_2$ to $x_2$. Deleting $xy$, $x_1y_1$ and $x_3y_3$ from $P_0$, $P_1$ and $P_3$, respectively, and joining $xy_1$, $yx_3$, $x_1y_2$ and $x_2y_3$ will lead to three new paths $P_0'$, $P_1'$ and $P_3'$, where $P_0'$ is from $s_0$ to $t_3$ via $yx_3$, $P_1'$ is from $s_1$ to $t_0$ via $y_1x$, and $P_3'$ is from $s_3$ to $t_1$ via $y_3x_2$, $P_2$ and $y_2x_1$. By deleting edges of $P_0$, $P_1$ and $P_3$ from $(2n-2)$-DPC of $BH_{n}-V(BH_{n-1}^2)$, and adding edges of $P_0'$, $P_1'$ and $P_3'$, a $(2n-2)$-DPC of $BH_{n}$ follows (see Fig. \ref{g3-1}).

\begin{figure}
\centering
\includegraphics[width=60mm]{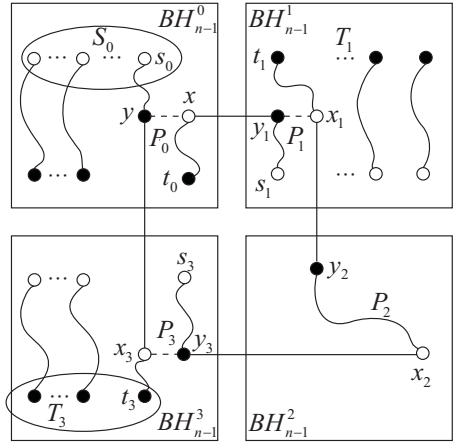}
\caption{Illustration of Case 2.1.} \label{g3-1}
\end{figure}

Finally, we claim that there exists an edge $xy\in E(BH_{n-1}^1)$ such that $y_1$ (resp. $x_3$) is not an endpoint of a path of $|T_{1}|$-DPC (resp. $|T_3|$-DPC) of $BH_{n-1}^1$ (resp. $BH_{n-1}^3$). Since neither $x$ nor $y$ is an endpoint of a path $P_0$ in $|S_0|$-DPC of $BH_{n-1}^0$, combing with the direction of $xy$ on $P_0$, there are at least $\lfloor\frac{4^{n-1}-2\ast(2n-4)}{2}\rfloor$ choices of $xy$ in $|S_0|$-DPC of $BH_{n-1}^0$. On the contrary, if both $y_1$ (resp. $x_3$) and its backup vertex are endpoints of two paths in $|T_{1}|$-DPC (resp. $|T_3|$-DPC) of $BH_{n-1}^1$ (resp. $BH_{n-1}^3$), then it will eliminate one choice of $xy$. Observe that $|T_1|+|T_3|\leq 2\ast(2n-2)$ and $2\lfloor\frac{4^{n-1}-2\ast(2n-4)}{2}\rfloor>2\ast(2n-2)$ whenever $n
\geq3$. Thus, the claim holds.

\noindent{\bf Case 2.2.} $j=1$. Analogous to the proof of Case 1.1, we can obtain a $(2n-2)$-DPC of $BH_{n}-V(BH_{n-1}^1)$.

Similarly, we can choose an appropriate edge $xy\in E(BH_{n-1}^3)$ on one of a path, say $P_3$, of $(2n-2)$-DPC of $BH_{n}-V(BH_{n-1}^1)$ such that $y_0$ (resp. $x_2$) is an internal vertex of a path $P_0$ (resp. $P_2$) of $|S_{0}|$-DPC (resp. $|T_2|$-DPC) in $BH_{n-1}^0$ (resp. $BH_{n-1}^2$), where $y_0$ and $x_2$ are $(n-1)$-dimensional neighbors of $x$ and $y$, respectively.

We claim that there exists an edge $xy\in E(BH_{n-1}^3)$ such that $y_0$ (resp. $x_2$) is not an endpoint of a path of $|S_{0}|$-DPC (resp. $|T_2|$-DPC) of $BH_{n-1}^0$ (resp. $BH_{n-1}^2$). Since neither $x$ nor $y$ is an endpoint of a path $P_3$ in $|T_3|$-DPC of $BH_{n-1}^3$, combing with the direction of $xy$ on $P_3$, there are at least $\lfloor\frac{4^{n-1}-2\ast(2n-2)}{2}\rfloor$ choices of $xy$ in $|T_3|$-DPC of $BH_{n-1}^3$. On the contrary, if both $y_0$ (resp. $x_2$) and its backup vertex are endpoints of two paths in $|S_{0}|$-DPC (resp. $|T_2|$-DPC) of $BH_{n-1}^0$ (resp. $BH_{n-1}^2$), then it will eliminate one choice of $xy$. Observe that $|S_0|+|T_2|\leq (2n-4)+(2n-2)$ and $2\lfloor\frac{4^{n-1}-2\ast(2n-2)}{2}\rfloor>(2n-4)+(2n-2)$ whenever $n\geq3$, thus, the claim holds.


We may assume that $s_3, y,x$ and $t_3$ lie sequentially on $P_3$, where $s_3$ and $t_3$ are the endpoints of $P_3$. Let $x_0$ (resp. $y_2$) be a neighbour of $y_0$ (resp. $x_2$) on $P_0$ (resp. $P_2$). Suppose without loss of generality that $s_0,y_0,x_0$ and $t_0$ (resp. $s_2,y_2,x_2$ and $t_2$) lie sequentially on $P_0$ (resp. $P_2$), where $s_0$ and $t_0$ (resp. $s_2$ and $t_0$) are the endpoints of $P_0$ (resp. $P_2$). Let $y_1$ and $x_1$ be $(n-1)$-dimensional neighbors of $x_0$ and $x_2$, respectively. By Lemma \ref{lace}, there exists a Hamiltonian path $P_1$ of $BH_{n-1}^1$ from $y_1$ to $x_1$. Deleting $xy$, $x_0y_0$ and $x_2y_2$ from $P_3$, $P_0$ and $P_2$ and joining $xy_0$, $yx_2$, $x_0y_1$ and $x_1y_2$ will lead to three new paths $P_0'$, $P_2'$ and $P_3'$, where $P_0'$ is from $s_0$ to $t_3$ via $y_0x$, $P_2'$ is from $s_2$ to $t_0$ via $y_2x_1$, $P_1$ and $y_1x_0$, and $P_3'$ is from $s_3$ to $t_2$ via $yx_2$. By deleting edges of $P_0$, $P_2$ and $P_3$ from $(2n-2)$-DPC of $BH_{n}-V(BH_{n-1}^1)$, and adding edges of $P_0'$, $P_2'$ and $P_3'$, a $(2n-2)$-DPC of $BH_{n}$ follows (see Fig. \ref{g3-2}).

\begin{figure}[h]
\centering
\includegraphics[width=60mm]{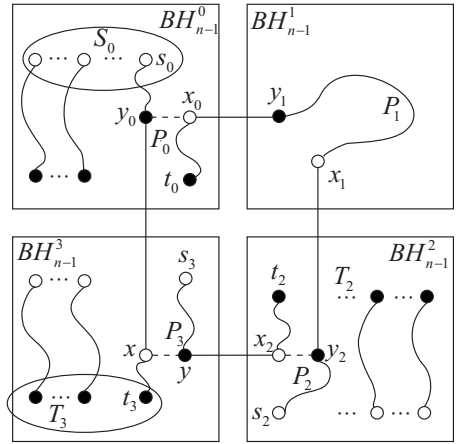}
\caption{Illustration of Case 2.2.} \label{g3-2}
\end{figure}

\noindent{\bf Case 3.} $S_i=\emptyset$ and $T_i=\emptyset$ and $S_j=\emptyset$ and $T_j=\emptyset$ for $i,j\in\{0,1,2,3\}$. By the relative positions of $BH_{n-1}^i$ and $BH_{n-1}^j$, there are two cases to consider.

\noindent{\bf Case 3.1.} $i=j+1$ or $j=i+1$. Suppose without loss of generality that $j=i+1$. There are two essentially distinct cases to consider.

\noindent{\bf Case 3.1.1.} $i=1$ and $j=2$. That is, $S_1=T_1=\emptyset$ and $S_2=T_2=\emptyset$. Note that $|S_0|\leq 2n-4$ and $|S_3|\leq 2n-4$. Additionally, $|T_0|\leq |S_0|$ and $|T_3|\geq |S_3|$. By the proof of Case 1.1, we can obtain $(2n-2)$-DPC of $BH_n-V(BH_{n-1}^1)\cup V(BH_{n-1}^2)$. Similarly, we shall make some changes to $(2n-2)$-DPC of $BH_{n}-V(BH_{n-1}^1)\cup V(BH_{n-1}^2)$, yielding a $(2n-2)$-DPC of $BH_n$.

\begin{figure}[h]
\centering
\includegraphics[width=60mm]{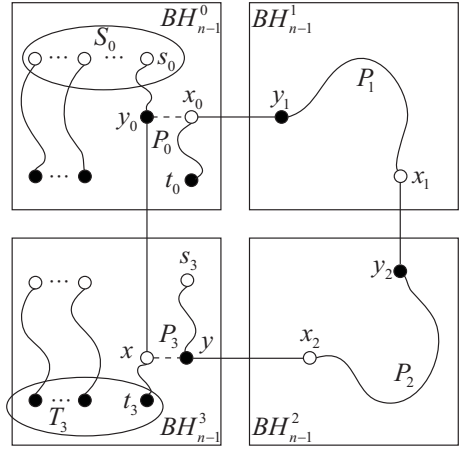}
\caption{Illustration of Case 3.1.1.} \label{g4}
\end{figure}

Suppose that $xy\in E(BH_{n-1}^3)$ is an edge on one of a path, say $P_3$, of $(2n-2)$-DPC of $BH_{n}-V(BH_{n-1}^1)\cup V(BH_{n-1}^2)$ such that $s_3,y,x$ and $t_3$ lie sequentially on $P_3$, where $s_3$ and $t_3$ are the endpoints of $P_3$. Let $y_0$ and $x_2$ be $(n-1)$-dimensional neighbors of $x$ and $y$, respectively. By the proof of Case 2.1, we may assume that $y_0$ is an internal vertex of a path $P_0$ of $|S_{0}|$-DPC in $BH_{n-1}^0$. Let $x_0$ be a neighbour of $y_0$ on $P_0$. Suppose without loss of generality that $s_0,y_0,x_0$ and $t_0$ lie sequentially on $P_0$, where $s_0$ and $t_0$ are the endpoints of $P_0$. Let $y_1$ be an $(n-1)$-dimensional neighbor of $x_0$ and let $x_1y_2$ be an edge from $BH_{n-1}^1$ to $BH_{n-1}^2$. By Lemma \ref{lace}, there exists a Hamiltonian path $P_1$ (resp. $P_2$) of $BH_{n-1}^1$ (resp. $BH_{n-1}^2$) from $y_1$ to $x_1$ (resp. $y_2$ to $x_2$). Deleting $xy$ and $x_0y_0$ from $P_0$ and $P_3$ and joining $xy_0$, $yx_2$, $x_0y_1$ and $x_1y_2$ will lead to two new paths $P_0'$ and $P_3'$, where $P_0'$ is from $s_0$ to $t_3$ via $y_0x$, and $P_3'$ is from $s_3$ to $t_0$ via $yx_2$, $P_2$, $y_2x_1$, $P_1$ and $y_1x_0$. By deleting edges of $P_0$ and $P_3$ from $(2n-2)$-DPC of $BH_{n}-V(BH_{n-1}^1)\cup V(BH_{n-1}^2)$, and adding edges of $P_0'$ and $P_3'$, a $(2n-2)$-DPC of $BH_{n}$ follows (see Fig. \ref{g4}).

\noindent{\bf Case 3.1.2.} $i=2$ and $j=3$. That is, $S_2=T_2=\emptyset$ and $S_3=T_3=\emptyset$. Note that $|S_0|\leq 2n-4$ and $|S_1|\leq 2n-4$. Additionally, $|T_0|\leq |S_0|$ and $|T_1|\geq |S_1|$.

If $|T_1|=|S_1|$, then $|T_1|\leq 2n-4$. The proof is analogous to that of Case 3.1.1.

If $|T_1|>|S_1|$ and $|T_1|\leq 2n-4$, then the proof is analogous to that of Condition (1) in Case 1.1.

If $|T_1|>2n-4$, then $|T_1|>|S_1|$. Then the proof is analogous to that of Condition (2) in Case 1.1.

\noindent{\bf Case 3.2.} $i=j+2$ or $j=i+2$. Suppose without loss of generality that $i=1$ and $j=3$. We further consider the following two cases

\noindent{\bf Case 3.2.1.} $|T_2|=|S_2|$. Then $|T_2|\leq 2n-4$. Additionally, $|T_0|=|S_0|$. By the induction hypothesis, we can obtain $|T_2|$-DPC and $|S_0|$-DPC of $BH_{n-1}^2$ and $BH_{n-1}^0$, respectively. Obviously, we can choose an edge $x_0y_0\in E(BH_{n-1}^0)$ on one of a path, say $P_0$, of $|S_0|$-DPC of $BH_{n-1}^0$ such that $s_0,y_0,x_0$ and $t_0$ lie sequentially on $P_0$, where $s_0$ and $t_0$ are the endpoints of $P_0$. Similarly, we can choose an edge $x_2y_2\in E(BH_{n-1}^2)$ on one of a path, say $P_2$, of $|T_2|$-DPC of $BH_{n-1}^2$ such that $s_2,y_2,x_2$ and $t_2$ lie sequentially on $P_2$, where $s_2$ and $t_2$ are the endpoints of $P_2$. Let $x_1,y_1,x_3$ and $y_3$ be $(n-1)$-dimensional neighbors of $y_2,x_0,y_0$ and $x_2$, respectively. By Lemma \ref{lace}, there exists a Hamiltonian path $P_1$ (resp. $P_3$) of $BH_{n-1}^1$ (resp. $BH_{n-1}^3$) from $y_1$ to $x_1$ (resp. $y_3$ to $x_3$). Deleting $x_0y_0$ and $x_2y_2$ from $P_0$ and $P_2$ and joining $x_0y_1$, $x_1y_2$, $x_2y_3$ and $x_3y_0$ will lead to two new paths $P_0'$ and $P_2'$, where $P_0'$ is from $s_0$ to $t_2$ via $y_0x_3$, $P_3$, $y_3x_2$, and $P_2'$ is from $s_2$ to $t_0$ via $y_2x_1$, $P_1$ and $y_1x_0$. By deleting edges of $P_0$ and $P_2$ from $|S_0|$-DPC of $BH_{n-1}^0$ and $|T_2|$-DPC of $BH_{n-1}^2$, and adding edges of $P_0'$ and $P_2'$, a $(2n-2)$-DPC of $BH_{n}$ follows (see Fig. \ref{g5}).

\begin{figure}[h]
\centering
\includegraphics[width=60mm]{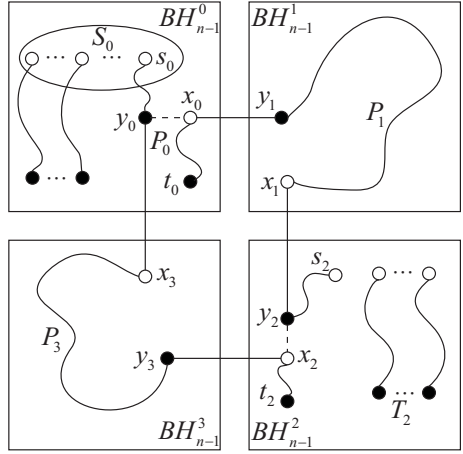}
\caption{Illustration of Case 3.2.1.} \label{g5}
\end{figure}

\noindent{\bf Case 3.2.2.} $|T_2|>|S_2|$. Analogous to the proof of Case 1.1, we can obtain $(2n-2)$-DPC of $BH_{n}-V(BH_{n-1}^1)$. The remaining proof is just a rewrite of that of Case 2.2, so we omit it.

\end{proof}

\section{Conclusions}
In this paper, unpaired many-to-many DPC of the balanced hypercube is obtained. We use induction to prove that the balanced hypercube $BH_n$, $n\geq2$, has unpaired ($2n-2$)-DPC. Let $u,u'\in V_0$ be two vertices having the same neighborhood and let $W=\{v_1,v_2,\cdots,v_{2n-1}\}$ be a set containing $2n-1$ distinct vertices. In addition, $u,u'\not\in S$ and $W\subseteq T$. If each vertex of $W$ is a neighbor of $u$ and $u'$, then one of $u$ and $u'$ can not be covered by any $(2n-1)$-DPC. That is, the number of disjoint paths in any unpaired many-to-many DPC can not exceed $2n-2$, indicating the upper bound $2n-2$ of the number of disjoint paths in ($2n-2$)-DPC is optimal.

It is meaningful to explore whether the upper bound $2n-2$ holds for paired many-to-many DPC. From the perspective of distributed computing, one may study algorithms to obtain many-to-many DPC in the balanced hypercube. Moreover, unpaired many-to-many DPC of the balanced hypercube with faulty elements is of interest and should be further investigated.

\vskip 0.3 in

%

\end{document}